\documentclass[preprint,review]{elsarticle}
\usepackage{amsmath}
\usepackage{amsfonts}
\usepackage{amssymb}
\usepackage{enumerate}
\usepackage{amsthm}
\usepackage{graphicx}
\usepackage{verbatim}
\newtheorem{myproposition}{Proposition}[section]
\newtheorem{mytheorem}[myproposition]{Theorem}
\newtheorem{mylemma}[myproposition]{Lemma}

\newtheorem{mycorollary}[myproposition]{Corollary}

\newtheorem{myobservation}[myproposition]{Observation}

\newtheorem{myproblem}[myproposition]{Problem}

\makeatletter
\def\imod#1{\allowbreak\mkern10mu({\operator@font mod}\,\,#1)}
\makeatother

\begin{document}

\begin{frontmatter}

\title{Spectra of Graphs and Closed Distance Magic Labelings}

\author[inst1]{Marcin Anholcer}
\author[inst2]{Sylwia Cichacz\fnref{fn1}}
\author[inst3]{Iztok Peterin\fnref{fn2}}

\address[inst1]{
Pozna\'n University of Economics, Faculty of Informatics and Electronic Economy, Al. Niepodleg{\l}o\'sci 10, 61-875 Pozna\'n, Poland, \textit{m.anholcer@ue.poznan.pl}
}
\address[inst2]{
AGH University of Science and Technology, Faculty of Applied Mathematics, Al. Mickiewicza 30, 30-059 Krak\'ow, Poland, \textit{cichacz@agh.edu.pl}
}
\address[inst3]{
University of Maribor, Faculty of Electrical Engineering and Computer Science, Smetanova 17, 2000 Maribor, Slovenia, \textit{iztok.peterin@um.si}
}

\fntext[fn1]{The author was partially supported by National Science Centre grant nr 2011/01/D/ST1/04104, as well as by the Polish Ministry of Science and Higher Education.}
\fntext[fn2]{The author is also with IMFM, Jadranska 19, 1000 Ljubljana, Slovenia and his work is partially supported by grant P1-0297 of Ministry of Education of Slovenia.}

\begin{abstract}
Let $G=(V,E)$ be a graph of order $n$. A closed distance magic labeling
of $G$ is a bijection $\ell \colon V(G)\rightarrow \{1,\ldots ,n\}$ for
which there exists a positive integer $k$ such that
$\sum_{x\in N[v]}\ell (x)=k$ for all $v\in V $, where $N[v]$ is the
closed neighborhood of $v$. We consider the closed distance magic graphs
in the algebraic context. In particular we analyze the relations between
the closed distance magic labelings and the spectra of graphs. These
results are then applied to the strong product of graphs with complete
graph or cycle and to the circulant graphs. We end with a number theoretic
problem whose solution results in another family of closed distance magic
graphs somewhat related to the strong product.
\end{abstract}

\end{frontmatter}

\section{Introduction and preliminaries}

All graphs considered in this paper are simple finite graphs. For a graph $G$, we use $V(G)$ for the vertex set and $E(G)$ for the edge set of $G$. The \textit{open neighborhood} $N(x)$ (or more precisely $N_{G}(x)$, when needed) of a vertex $x$ is the set of all vertices adjacent to $x$, and the \textit{degree} $d(x)$ of $x$ is $|N(x)|$, i.e. the size of the neighborhood of $x$. By $N[x]$ (or $N_{G}[x]$) we denote the \textit{closed neighborhood} $N(x)\cup\{x\}$ of $x$. By $C_{n}$ we denote a cycle on $n$ vertices.\\

Different kinds of labelings have been important part of graph theory. See a dynamic survey \cite{Ga} which covers the field. One type of labelings includes magic labelings, where some elements (edges, vertices, etc.) of a graph must be labeled in such a way that certain sums (depending on graph properties) are constant. \emph{Closed distance magic labeling} (also called $\Sigma ^{\prime }$ -\textit{labeling}, see \cite{Be}) of a graph $G=(V(G),E(G))$ of order $n$ is a bijection $\ell\colon V(G)\rightarrow \{1,\ldots ,n\}$ with the property that there is a positive integer $k^{\prime }$ (called the \emph{magic constant}) such that $w(x)=\sum_{y\in N_{G}[x]}\ell (y)=k^{\prime }$ for every $x\in V(G)$, where $w(x)$ is the \emph{weight} of $x$. If a graph $G$ admits a closed distance magic labeling, then we say that $G$ is \emph{closed distance magic graph}. Closed distance magic graphs are an analogue to distance magic graphs, where the sums are taken over the open neighborhoods $N_{G}(x)$ instead of the closed ones $N_{G}[x]$, see \cite{AFK,Cic,CicFro}.

Let $D$ be a subset of non-negative integers. O'Neal and Slater in \cite{ONSl2} have defined the $D$-distance magic labeling as a bijection $f:V(G)\rightarrow \{1,\ldots,n\}$ such that there is a magic constant $k$ such that for any vertex $x\in V(G)$, $w(x)=\sum_{y\in N_D(x)}{f(y)}=k$. Here $N_D(x)=\{y\in V(G)|d(x,y)\in D\}$, i.e., the weight of a vertex $x\in V(G)$ is the sum of the labels of all the vertices $y\in V(G)$ for which their distance to $x$ belongs to $D$. This kind of labeling has been studied e.g. by Simanjuntak et al. in \cite{ref_SimElv} (we refer to some of their results in one of the following sections). It is a generalization of both distance magic labeling and closed distance magic labeling, where $D=\{1\}$ and $D=\{0,1\}$, respectively.

The concept of distance magic labeling has been motivated by the construction of magic rectangles. Magic rectangles are natural generalization of magic squares that has been intriguing mathematicians and the general public for a long time \cite{Hag}. A magic $(m,n)$-rectangle $S$ is an $m\times n$ array in which the first $mn$ positive integers are placed so that the sum over each row of $S$ is constant and the sum over each column of $S$ is another (different if $m\neq n$) constant. Harmuth proved that:

\begin{mytheorem}[\protect\cite{Har1,Har2}]
\label{rectangles} For $m, n>1$ there is a magic $(m, n)$-rectangle $S$ if and only if $m \equiv n \, ({\rm mod}\, 2)$ and $(m, n) \neq (2, 2)$.
\end{mytheorem}

A related concept is the notion of \emph{distance antimagic labeling}. This is again a bijection $\overline{f}$ from $V(G)$
to $\{1,\ldots ,n\}$ but this time different vertices are required to have distinct weights (where the sums are taken over the open neighborhoods $N_{G}(x)$). A more restrictive
version of this labeling is the $(a, d)$-distance antimagic labeling. It is a distance antimagic labeling with
the additional property that the weights of vertices form an arithmetic progression with difference $d$ and
first term $a$. If $d = 1$, then $\overline{f}$ is called simply \emph{distance antimagic labeling} \cite{Fro}. Notice that if  a graph has a closed distance magic labeling then it has a distance antimagic labeling. The opposite is however not true, as it can be easily checked on the example of $C_5$. It has no closed distance magic labeling, while it has a distance antimagic labeling $\overline{f}(v_i)=i$, where $V(C_5)=\{v_1,v_2,v_3,v_4,v_5\}$ and $E(C_5)=\{\{v_i,v_j\}:|i-j|=1 \vee |i-j|=4\}$.

Finding an $r$-regular distance antimagic labeling
turns out to be equivalent to finding  a fair incomplete tournament FIT$(n, r)$
\cite{Fro}. A \emph{fair incomplete tournament} of $n$ teams with $g$ rounds, FIT$(n,r)$, is a tournament
in which every team plays $r$ other teams and the total strength of the opponents that
team $i$ plays is $S_{n,r}(i) = (n + 1)(n - 2)/2 + i - c$ for every $i$ and some fixed constant
$c$.

We recall one of four standard graph products (see \cite{IK}). Let $G$ and $H$ be two graphs. The \emph{strong product} $G\boxtimes H$ is a graph with vertex set $V(G)\times V(H)$. Two vertices $(g,h)$ and $(g',h')$ are adjacent in $G\boxtimes H$ if either $g=g'$ and $h$ is adjacent with $h'$ in $H$, or $h=h'$ and $g$ is adjacent with $g'$ in $G$, or $g$ is adjacent with $g'$ in $G$ and $h$ is adjacent with $h'$ in $H$. Recently in \cite{ACPT} two other standard products, namely direct and lexicographic, have been considered with respect to the property of being distance magic.

It is easy to notice the following observation, that will be useful in our further considerations.

\begin{myobservation}
\label{cor_regular} If $G$ is an $r$-regular closed distance magic graph on $n$ vertices, then $k'= \frac{(r+1)(n+1)}{2}$.
\end{myobservation}

In the next section we reveal somewhat surprising connection between the existence of closed distance magic labeling and the spectrum of a graph.
In the following sections we consider the existence of closed distance magic labelings of chosen families of graphs
using algebraic tools developed in Section \ref{algeb}. In the final section we present a combinatorial problem whose
solution yields more closed distance magic graphs.

\section{Necessary conditions - algebraic approach}\label{algeb}

Let $V(G)=\{x_{1},\dots ,x_{n}\}$. Then the following system of equations with unknowns $\ell (x_{1}),\dots ,\ell (x_{n})$ has to be satisfied for every closed distance magic graph:
\begin{equation}\label{system1}
\begin{array}{l}
w(x_{1})=k', \\
w(x_{2})=k', \\
\quad \vdots \\
w(x_{n})=k'.
\end{array}%
\end{equation}

By writing this system in matrix form, we get
\begin{equation*}
(\textbf{A}(G)+\textbf{I}_{n})\textbf{l}=k' \textbf{u}_n,
\end{equation*}
where $\textbf{A}(G)$ is the adjacency matrix of $G$, $\textbf{I}_{n}$ is $n\times n$ identity matrix, $\textbf{l}=(\ell(x_{1}),\ldots ,$ $\ell(x_{n}))$ and $\textbf{u}_n$ is a vector of length $n$ with every entry equal to $1$.

It is well known that the rank of a square matrix is equal to the number of its non-zero singular values (see e.g. \cite{Chat}, p.31). In the case of symmetric matrices it is in turn equal to the number of the non-zero eigenvalues. It means that the dimension of the set of solutions of the system (\ref{system1}) equals to the multiplicity of $0$ in the spectrum of $\textbf{A}(G)+\textbf{I}_{n}$, i.e., to the multiplicity of $-1$ in $Sp(G)$, where $Sp(G)$ denotes the spectrum of $G$.

It has been recently proved by O'Neil and Slater in \cite{ONSl} that if a graph is closed distance magic, then the magic constant $k^\prime$ is unique, i.e., even if there exist two distinct closed distance magic labelings, then they result in same magic constant $k^\prime$.

The above considerations lead us to the following result.

\begin{mytheorem}
\label{thm_spectrum}If $G$ is a closed distance magic graph and the system (\ref{system1}) has $k+1$ linearly independent solutions, then the multiplicity of $-1$ in $Sp(G)$ is $k$.
\end{mytheorem}

The following corollary will be used in the remainder of the paper.

\begin{mycorollary}\label{piki}
Let $G$ be a closed distance magic graph such that there exist $k$ linearly independent solutions to the system (\ref{system1}) such that no bijection $\ell\colon V(G)\rightarrow \{1,\dots,|V(G)|\}$ is their linear combination. Then the multiplicity of $-1$ in $Sp(G)$ is at least $k$.
\end{mycorollary}

If $G$ is an $r$-regular graph, then clearly the system (\ref{system1}) has at least one solution not being bijection from $V(G)$ to $\{1,\dots,|V(G)|\}$, namely $\ell (x_{1})=\dots =\ell (x_{n})=k^{\prime }/(r+1)$. Hence the following holds:

\begin{mycorollary}\label{neces_Sp-1}
\label{prop_spectrum}If $G$ is a regular closed distance magic graph, then $-1\in Sp(G)$.
\end{mycorollary}

A \textit{perfect code} is a subset $C(G)$ of $V(G)$ such that the closed neighborhoods of the vertices $v\in C$ form a partition of $G$. It is known that every regular graph $G$ with a perfect code must satisfy $-1\in Sp(G)$ (see e.g. \cite{Big}, p. 22). Observe however that in the case of a regular graph $G$ having a perfect code, the system (\ref{system1}) has at least the following two solutions: $\ell (x_{1})=\dots =\ell (x_{n})=k'/(r+1)$ and $\ell(x)=k', x\in C(G)$, $\ell(x)=0, x\in V(G)-C(G)$. Obviously these solutions are linearly independent and no bijection $\ell\colon V(G)\rightarrow \{1,\dots,|V(G)|\}$ is their linear combination. This means that the following is true.

\begin{mycorollary}
Let $G$ be a regular closed distance magic graph having a perfect code. Then the multiplicity of $-1$ in $Sp(G)$ is at least $2$.
\end{mycorollary}

The following fact can be found, e.g., in \cite{ref_BroHae}, p.11.

\begin{myobservation}
Given any graphs $G$ and $H$, the eigenvalues of $G\boxtimes H$ have the form $(\lambda_G+1)(\lambda_H+1)-1$, where $\lambda_G\in Sp(G)$ and $\lambda_H\in Sp(H)$.
\end{myobservation}

It follows that the value $-1$ can appear in $Sp(G\boxtimes H)$ if and only if it appears in the spectrum of at least one of the graphs $G$ and $H$. This leads us to the following corollary.

\begin{mycorollary}
If graphs $G$ and $H$ are regular and $G\boxtimes H$ is closed distance magic, then $-1\in Sp(G)\cup Sp(H)$.
\end{mycorollary}

If $G$ is an $r$-regular graph with $s$ distinct eigenvalues $r, \lambda_1,\dots,\lambda_{s-1}$, then its line graph $L(G)$ has at most $s+1$ distinct eigenvalues $2r-2, r+\lambda_1-2,\dots,r+\lambda_{s-1}-2,-2$. So the following is true.

\begin{mycorollary}
Let $G$ be an $r$-regular graph, $r>1$. If its line graph $L(G)$ is closed distance magic, then $1-r\in Sp(G)$.
\end{mycorollary}

In the following sections we are going to discuss the existence of closed distance magic labelings of certain families of graphs.


\section{Complete graphs and their strong products}

It is obvious that every complete graph is closed distance magic (observe that every bijection $\ell \colon V(G)=K_n\rightarrow \{1,\ldots ,n\}$ results in equal vertex weights). This is however not true in the case of the complete bipartite graphs. Corollary \ref{neces_Sp-1} may be generalized on some non-regular graphs. It is enough that we are able to find any labeling satisfying the system of equations (\ref{system1}), not being the bijection from $V(G)$ to $\{1,\dots ,n\}$. Thus, for example, the statement in Corollary \ref{prop_spectrum} remains true also for complete bipartite graphs $K_{m,n}$, where $2\leq m<n$. Here for $V(K_{m,n})=\{x_{1},\ldots ,x_{n}\}\cup \{x_{n+1},\ldots ,x_{n+m}\}$, the sample solution of the system (\ref{system1}) is $\ell (x_{1})=\dots =\ell (x_{n})=k^{\prime }/m$ and $\ell (x_{n+1})=\dots =\ell (x_{n+m})=k^{\prime }/n$. This leads us immediately to the following result.

\begin{myproposition}
For any $m$ and $n$ such that $2\leq m\leq n$, $K_{m,n}$ is not closed distance magic.
\end{myproposition}

\begin{proof}The spectrum of the graph $K_{m,n}$ is $Sp(K_{m,n})=\{-\sqrt{mn},0^{m+n-2},\sqrt{mn}\}$ (see \cite{ref_BroHae}, p.8.), thus $-1$ is never its element for $m,n\geq 2$. In consequence, for any value $k^\prime$, the only solution of the system (\ref{system1}) is $\ell (x_{1})=\dots =\ell (x_{n})=k^{\prime }/m$ and $\ell (x_{n+1})=\dots =\ell (x_{n+m})=k^{\prime }/n$.
\end{proof}

The following result related to strong product was proved by Beena.

\begin{mytheorem}[\protect\cite{Be}]
\label{Beena} Let $G$ be any connected $r$-regular graph, $r > 0$. If $n$ is even, then $K_n\boxtimes G$ is a closed distance magic graph.
\end{mytheorem}

Next result extends Theorem~\ref{Beena} for the case when both graphs have odd order.

\begin{mytheorem}\label{magic_product_G_K}
Let $n$ be an odd number and $r>0$. If $G$ is an $r$-regular graph of an odd order, then $K_{n}\boxtimes G$ is a closed distance magic graph.
\end{mytheorem}

\begin{proof} Let  $V(K_{n})=\{y_{0},\ldots ,y_{n-1}\}$ and $V(G)=\{x_{0},\ldots ,x_{m-1}\}$ for odd $m$ and $n$. For a vertex $(y_{i},x_{j})$ from $V(K_{n}\boxtimes G)$ we simply write $v_{i,j}$. Notice that if $x_{p}x_{q}\in E(G)$, then $v_{j,q}\in N_{H}(v_{i,p})$ for every $j\in \{0,\ldots ,n-1\}$. There exists a magic $(n,m)$-rectangle by Theorem~\ref{rectangles}. Let $a_{i,j}$ be an $(i,j)$-entry of the $(n,m)$-rectangle, $1\leq i\leq n$ and $1\leq j\leq m$ . Define the labeling $\ell \colon V(K_{n}\boxtimes G)\rightarrow \{1,\ldots ,nm\}$ as $\ell(v_{i,j})=a_{i+1,j+1}$ for $i\in \{0,\ldots ,n-1\}$ and $j\in \{0,\ldots,m-1\}$. Obviously $\ell $ is a bijection and moreover for each $j$ we have $\sum_{i=0}^{n-1}\ell(v_{i,j})=C$ for some constant $C$. Therefore for any $x\in V(K_{n}\boxtimes G)$ we have $w(x)=(r+1)C=k^\prime$.\end{proof}

An $r$-regular graph $G$ is called $(r,a,b)$-strongly regular if every pair of adjacent vertices has $a\geq 0$ common neighbors and every pair of non-adjacent vertices has $b\geq 1$ common neighbors. The following is true.

\begin{myproposition}
If a strongly regular graph $G$ on $n$ vertices is closed distance magic, then $G\cong K_n$.
\end{myproposition}


\begin{proof} The eigenvalues of $G$ are $r$ and two roots $x_1$, $x_2$ of the equation $x^2+(c-a)x+c-r=0$ (see e.g. \cite{Big}, p. 20). By Corollary \ref{neces_Sp-1}, $-1$ must be a solution of this equation and thus $a=r-1$. This implies that $G$ is a disjoint union of some number of copies of $K_n$. In such a graph, the non-adjacent vertices do not have common neighbors. As in the strongly regular graph every pair of non-adjacent vertices has at least one vertex in common, we have $G\cong K_n$.\end{proof}


\section{Cycles and their strong products}

It is easy to see that $C_3$ is the only closed distance magic graph among cycles.
Nevertheless, we start this section with (also easy) algebraic proof of this fact in order to present a simple application of the methods developed in Section \ref{algeb}.
For this we need the following observation, which can be found for instance in \cite{ref_BroHae}, p.9.

\begin{myobservation}
The spectrum of an undirected cycle $C_{n}$ consists of numbers $2\cos(2\pi j/n)$, $j\in \{1,\dots ,n\}$.
\end{myobservation}

If $-1\in Sp(C_n)$, then we have for some $j$
\begin{equation*}
2\cos(2\pi j/n)=-1,
\end{equation*}
what is equivalent to the fact that $j\in \{n/3, 2n/3\}$. It implies in
turn that $n\equiv 0\, ({\rm mod}\, 3)$. Moreover, as $j$ can take one of
two values, it means that in this case the multiplicity of $-1$ in $Sp(C_n)$
is exactly $2$. This leads us to the following results.

\begin{myproposition}\label{miki}
The only closed distance magic cycle is $C_3=K_3$.
\end{myproposition}

\begin{proof}
Assume $C_n$ is closed distance magic with magic constant $k^\prime$ and
denote its vertices by $x_0$, $x_1$, $\dots$, $x_{n-1}$. We can assume that
$n\equiv 0\, ({\rm mod}\, 3)$. Let us consider the following three labelings
of $V(G)$. For every $i\in \{0,1,2\} $ and $j\in \{0,\ldots ,n-1\}$ set
$$
\ell_i(x_j)=\left\{
\begin{array}{lll}
k', &\text{if} & j\equiv i\, ({\rm mod}\, 3),\\
0, &\text{if} & j\not\equiv i\, ({\rm mod}\, 3).
\end{array}
\right.
$$
Obviously these labelings satisfy the system (\ref{system1}) and are linearly
independent. Moreover for any linear combination $\ell$ of $\ell_0$, $\ell_1$
and $\ell_2$, we have $\ell(x_{j_1})=\ell(x_{j_2})$ for $j_1\equiv j_2\, ({\rm mod}\, 3)$,
so there is no bijection $\ell \colon V(G)\rightarrow \{1,\ldots ,n\}$ linearly
independent from $\ell_0$, $\ell_1$ and $\ell_2$ if $n>3$. As the multiplicity
of $-1$ in $Sp(C_n)$ is $2$, the cycle is not closed distance magic if $n>3$.
\end{proof}

The following result is a direct consequence of Corollary \ref{piki} and the
discussion before Proposition \ref{miki}.

\begin{mycorollary}\label{cor_3Sp-1}
If the product $G\boxtimes C_n$ is closed distance magic, where $G$ is a regular
graph, then $n\equiv 0\, ({\rm mod}\, 3)$ or $-1\in Sp(G)$.
\end{mycorollary}

\begin{mycorollary}\label{cor_oddOdd}
Let the product $G\boxtimes C_n$ be closed distance magic, where $G$ is a $2d$-regular
graph on $m$ vertices. Then $n\equiv 1\, ({\rm mod}\, 2)$ and $m\equiv 1\, ({\rm mod}\, 2)$.
Moreover, $n\equiv 3\, ({\rm mod}\, 6)$ or $-1\in Sp(G)$.
\end{mycorollary}

\begin{proof} Let $k^\prime$ be the magic constant of $G\boxtimes C_n$. The strong product
of $r_1$- and $r_2$-regular graphs is $((r_1+1)(r_2+1)-1)$-regular, so by Observation
\ref{cor_regular} we have
\begin{equation*}
k^\prime=3(2d+1)(mn+1)/2.
\end{equation*}
This means that $m\equiv 1\, ({\rm mod}\, 2)$ and $n\equiv 1\, ({\rm mod}\, 2)$. Now if
$-1\not\in Sp(G)$, then by Corollary \ref{cor_3Sp-1}, $n\equiv 0\, ({\rm mod}\, 3)$
and in consequence $n\equiv 3\, ({\rm mod}\, 6)$. \end{proof}

Let $V(C_{m}\boxtimes C_{n})=\{v_{i,j}:0\leq i\leq m-1,0\leq j\leq n-1\}$, where
$$N(v_{i,j})=\{v_{i-1,j-1},v_{i-1,j},v_{i-1,j+1},v_{i,j-1},v_{i,j+1},v_{i+1,j-1},v_{i+1,j},v_{i+1,j+1}\}$$
and arithmetics on the first suffix is taken modulo $m$ and on the second suffix modulo $n$.
We also refer to the set of all vertices $v_{i,j}$ with fixed $i$ as $i$-th row and with
fixed $j$ as $j$-th column. Below we give the necessary and sufficient conditions for
$C_{m}\boxtimes C_{n}$ to be closed distance magic.

\begin{mytheorem}
\label{conjecture}The strong product $C_{m}\boxtimes C_{n}$ is closed distance magic
if and only if at least one of the following conditions holds:

\begin{enumerate}
\item $m\equiv 3\, ({\rm mod}\, 6)$ and $n\equiv 3\, ({\rm mod}\, 6)$.

\item $\{m,n\}=\{3,x\}$ and $x$ is an odd number.
\end{enumerate}
\end{mytheorem}

\begin{proof}\
The following Lemma gives the necessary conditions for $C_{m}\boxtimes C_{n}$ to
be closed distance magic.

\begin{mylemma}
\label{notdm}If $C_{m}\boxtimes C_{n}$ is a closed distance magic graph, then at
least one of the following conditions holds:

\begin{enumerate}
\item $m\equiv 3\, ({\rm mod}\, 6)$ and $n\equiv 3\, ({\rm mod}\, 6)$.
\item $\{m,n\}=\{3,x\}$ and $x$ is an odd number.
\end{enumerate}
\end{mylemma}

\begin{proof} Assume that $C_{m}\boxtimes C_{n}$ is a closed distance magic graph with some magic constant $k^{\prime }$. Hence there exists a closed distance magic labeling $\ell :V(C_{m}\boxtimes C_{n})\rightarrow \{1,\ldots ,mn\}$. By Corollary \ref{cor_oddOdd}, $m$ and $n$ are odd. Assume that $m\not\equiv 3\, ({\rm mod}\, 6)$.

Let us consider the following $2m+1$ solutions of the system (\ref{system1}). For $i\in\{0,\dots,m-1\}$, $j\in\{0,\dots,n-1\}$ and $s\in\{1,\dots,m\}$ set

$$
\ell_s(x_{ij})=\left\{
\begin{array}{ll}
k'/3 & \text{if } i=s \wedge j\equiv 0\, ({\rm mod}\, 3),\\
k'/3 & \text{if } i\neq s \wedge j\equiv 1\, ({\rm mod}\, 3),\\
0 & \text{otherwise},
\end{array}
\right.
$$

$$
\ell_{m+s}(x_{ij})=\left\{
\begin{array}{ll}
k'/3 & \text{if } i=s \wedge j\equiv 1\, ({\rm mod}\, 3),\\
k'/3 & \text{if } i\neq s \wedge j\equiv 2\, ({\rm mod}\, 3),\\
0 & \text{otherwise},
\end{array}
\right.
$$

$$
\ell_{2m+1}(x_{ij})=\left\{
\begin{array}{ll}
k'/3 & \text{if } j\equiv 0\, ({\rm mod}\, 3),\\
0 & \text{otherwise}.
\end{array}
\right.
$$

The first $2m$ solutions are linearly independent as the value of $\ell(x_{ij})$ for $i\in\{0,\dots,m\}$, $j\equiv x\, ({\rm mod}\, 3)$, $x\in\{0,1\}$, is not equal to $0$ only in the solution $\ell_{i+(j\mod 3)m}$. The last solution is not a linear combination of the other ones as in order to obtain $\ell(x_{ij})=k^\prime/3$ for $j\equiv 0\, ({\rm mod}\, 3)$ and $\ell(x_{ij})=0$ for $j\equiv 1\, ({\rm mod}\, 3)$ we should have $\ell_{2m+1}(x)=\sum_{i=1}^{m}{(\ell_i(x)-(m-1)\ell_{m+i}(x)}$ for every $x$, which obviously does not result with $\ell(x_{ij})=0$ for $j\equiv 2\, ({\rm mod}\, 3)$. Moreover no bijection $\ell \colon V(G)\rightarrow \{1,\ldots ,n\}$ can be a linear combination of $\ell_i$, $i\in\{1,\dots,2m+1\}$ for $n>3$ as for any $\alpha_1,\dots,\alpha_{2m+1}$ we have $\sum_{s=1}^{2m+1}{\alpha_s\ell_s(x_{ij_1})}=\sum_{s=1}^{2m+1}{\alpha_s\ell_s(x_{ij_2})}$ for $j_1\equiv j_2\, ({\rm mod}\, 3)$. As the multiplicity of $-1$ in $Sp(G)$ is $2m$, $G=C_m\boxtimes C_n$ is not closed distance magic. This implies that $n=3$. Moreover $m$ is an odd number as noted before.
\end{proof}

The sufficiency of the second condition follows from Theorem \ref{magic_product_G_K}, as $C_3\cong K_3$. In order to finish the proof of the theorem, we are going to show that there exists a closed distance magic labeling of $C_{m}\boxtimes C_{n}$ for any $m\equiv 3\, ({\rm mod}\, 6)$ and $n\equiv 3\, ({\rm mod}\, 6)$.

First let us observe the following.

\begin{myobservation}\label{triples_sum}
For every $m\equiv 3 \, ({\rm mod}\, 6)$ one can divide the set $\{1,\dots,m\}$ into $m/3$ mutually disjoint triples such that the sum of the elements of each triple equals to $3(m+1)/2$.
\end{myobservation}

\begin{proof} The desired triples are $(2i+1,\frac{m+1}{2}-i,m-i)$ for $i\in \{0,\dots,\frac{m-3}{6}\}$ and $(2i+2,\frac{2m}{3}-i,\frac{5m-3}{6}-i)$ for $i\in\{0,\dots,\frac{m-9}{6}\}$.
\end{proof}

Let us prove the existence of the desired labeling.

\begin{mylemma}
\label{suff_cond}If $m\equiv 3\, ({\rm mod}\, 6)$ and $n\equiv 3\, ({\rm mod}\, 6)$, then $C_{m}\boxtimes C_{n}$ is a closed distance magic graph.
\end{mylemma}

\begin{proof} Let us denote the triples granted by Observation \ref{triples_sum} for the set $\{1,\dots,m\}$ by $S_0,\dots,S_{m/3-1}$ and the elements of a triple $S_p$, $p\in\{0,\dots, m/3-1\}$, by $s_p^0$, $s_p^1$ and $s_p^2$. Similarly, let the triples for the set $\{1,\dots,n\}$ be $T_0,\dots,T_{n/3-1}$ and denote the elements of a triple $T_q$, $q\in\{0,\dots, n/3-1\}$, by $t_q^0$, $t_q^1$ and $t_q^2$.

Let us define two following labelings of $C_{m}\boxtimes C_{n}$:
$$
\ell_1(v_{ij})=s_{\lfloor i/3\rfloor}^{j\, ({\rm mod}\, 3)},i\in\{0,\ldots,m-1\}, j\in\{0,\ldots,n-1\},
$$

$$
\ell_2(v_{ij})=t_{\lfloor j/3\rfloor}^{i\, ({\rm mod}\, 3)},i\in\{0,\ldots,m-1\}, j\in\{0,\ldots,n-1\}.
$$

In other words, in order to construct the labeling $\ell_1$ we put the elements of $S_0$ on consecutive triples of vertices in row $0$, and then we label rows $1$ and $2$ in the same way. To label rows $3$, $4$ and $5$ we use $S_1$, next three rows we label with $S_2$ and so on. Similarly the triples $T_0, T_1,\ldots$ are used to label the columns of $C_{m}\boxtimes C_{n}$. As an example, below we present the labelings $\ell_1$ and $\ell_2$ in the case of $C_{15}\boxtimes C_9$. In this case the partition of $\{1,2,\dots,15\}$ is $(1,8,15)$, $(3,7,14)$, $(5,6,13)$, $(2,10,12)$, $(4,9,11)$ and the partition of $\{1,2,\dots,9\}$ is $(1,5,9)$, $(3,4,8)$, $(2,6,7)$.

\begin{equation*}
\begin{array}{ccccccccc}
\ell _{1}: &  &  &  &  &  &  &  &  \\
1 & 8 & 15 & 1 & 8 & 15 & 1 & 8 & 15 \\
1 & 8 & 15 & 1 & 8 & 15 & 1 & 8 & 15 \\
1 & 8 & 15 & 1 & 8 & 15 & 1 & 8 & 15 \\
3 & 7 & 14 & 3 & 7 & 14 & 3 & 7 & 14 \\
3 & 7 & 14 & 3 & 7 & 14 & 3 & 7 & 14 \\
3 & 7 & 14 & 3 & 7 & 14 & 3 & 7 & 14 \\
5 & 6 & 13 & 5 & 6 & 13 & 5 & 6 & 13 \\
5 & 6 & 13 & 5 & 6 & 13 & 5 & 6 & 13 \\
5 & 6 & 13 & 5 & 6 & 13 & 5 & 6 & 13 \\
2 & 10 & 12 & 2 & 10 & 12 & 2 & 10 & 12 \\
2 & 10 & 12 & 2 & 10 & 12 & 2 & 10 & 12 \\
2 & 10 & 12 & 2 & 10 & 12 & 2 & 10 & 12 \\
4 & 9 & 11 & 4 & 9 & 11 & 4 & 9 & 11 \\
4 & 9 & 11 & 4 & 9 & 11 & 4 & 9 & 11 \\
4 & 9 & 11 & 4 & 9 & 11 & 4 & 9 & 11%
\end{array}%
\qquad
\begin{array}{ccccccccc}
\ell _{2}: &  &  &  &  &  &  &  &  \\
1 & 1 & 1 & 3 & 3 & 3 & 2 & 2 & 2 \\
5 & 5 & 5 & 4 & 4 & 4 & 6 & 6 & 6 \\
9 & 9 & 9 & 8 & 8 & 8 & 7 & 7 & 7 \\
1 & 1 & 1 & 3 & 3 & 3 & 2 & 2 & 2 \\
5 & 5 & 5 & 4 & 4 & 4 & 6 & 6 & 6 \\
9 & 9 & 9 & 8 & 8 & 8 & 7 & 7 & 7 \\
1 & 1 & 1 & 3 & 3 & 3 & 2 & 2 & 2 \\
5 & 5 & 5 & 4 & 4 & 4 & 6 & 6 & 6 \\
9 & 9 & 9 & 8 & 8 & 8 & 7 & 7 & 7 \\
1 & 1 & 1 & 3 & 3 & 3 & 2 & 2 & 2 \\
5 & 5 & 5 & 4 & 4 & 4 & 6 & 6 & 6 \\
9 & 9 & 9 & 8 & 8 & 8 & 7 & 7 & 7 \\
1 & 1 & 1 & 3 & 3 & 3 & 2 & 2 & 2 \\
5 & 5 & 5 & 4 & 4 & 4 & 6 & 6 & 6 \\
9 & 9 & 9 & 8 & 8 & 8 & 7 & 7 & 7%
\end{array}%
\end{equation*}

Observe that for every two vertices $v_{i_1j_1}\neq v_{i_2j_2}$, we have $(\ell_1(v_{i_1j_1}),\ell_2(v_{i_1j_1}))\neq (\ell_1(v_{i_2j_2}),\ell_2(v_{i_2j_2}))$.
Indeed, if $(\ell_1(v_{i_1j_1}),\ell_2(v_{i_1j_1}))= (\ell_1(v_{i_2j_2}),\ell_2(v_{i_2j_2}))$, then we have $s_{\lfloor i_1/3\rfloor}^{j_1\, ({\rm mod}\, 3)}=s_{\lfloor i_2/3\rfloor}^{j_2\, ({\rm mod}\, 3)}$ and $t_{\lfloor j_1/3\rfloor}^{i_1\, ({\rm mod}\, 3)}=t_{\lfloor j_2/3\rfloor}^{i_2\, ({\rm mod}\, 3)}$. This implies in turn that $\lfloor i_1/3\rfloor=\lfloor i_2/3\rfloor$ and $i_1\, ({\rm mod}\, 3)=i_2\, ({\rm mod}\, 3)$ and finally $i_1=i_2$. Similarly we can deduce that $j_1=j_2$.

Let us define the labeling $\ell:V(C_{m}\boxtimes C_{n})\rightarrow\{1,\dots,mn\}$ as
$$
\ell(x)=(\ell_1(x)-1)n+\ell_2(x).
$$

It is straightforward to see that $\ell$ is a bijection. The labeling $\ell$ of $C_{15}\boxtimes C_{9}$ is given below.

$$
\begin{array}{ccccccccc}
\ell &&&&&&&&\\
1 & 64 & 127 & 3 & 66 & 129 & 2 & 65 & 128\\
5 & 68 & 131 & 4 & 67 & 130 & 6 & 69 & 132\\
9 & 72 & 135 & 8 & 71 & 134 & 7 & 70 & 133\\
19 & 55 & 118 & 21 & 57 & 120 & 20 & 56 & 119\\
23 & 59 & 122 & 22 & 58 & 121 & 24 & 60 & 123\\
27 & 63 & 126 & 26 & 62 & 125 & 25 & 61 & 124\\
37 & 46 & 109 & 39 & 48 & 111 & 38 & 47 & 110\\
41 & 50 & 113 & 40 & 49 & 112 & 42 & 51 & 114\\
45 & 54 & 117 & 44 & 53 & 116 & 43 & 52 & 115\\
10 & 82 & 100 & 12 & 84 & 102 & 11 & 83 & 101\\
14 & 86 & 104 & 13 & 85 & 103 & 15 & 87 & 105\\
18 & 90 & 108 & 17 & 89 & 107 & 16 & 88 & 106\\
28 & 73 & 91 & 30 & 75 & 93 & 29 & 74 & 92\\
32 & 77 & 95 & 31 & 76 & 94 & 33 & 78 & 96\\
36 & 81 & 99 & 35 & 80 & 98 & 34 & 79 & 97
\end{array}
$$

On the other hand observe that for every $x\in V(C_{m}\boxtimes C_{n})$ we have
$$
\sum_{y\in N[x]}{\ell_1(y)}=3(s_0^0+s_0^1+s_0^2)=9(m+1)/2
$$
and
$$
\sum_{y\in N[x]}{\ell_2(y)}=3(t_0^0+t_0^1+t_0^2)=9(n+1)/2.
$$

It means that for every $x\in V(C_{m}\boxtimes C_{n})$
$$
\begin{array}{l}
\sum_{y\in N[x]}{\ell(y)}=\sum_{y\in N[x]}(\ell_1(y)-1)n+\ell_2(y)\\
=n\sum_{y\in N[x]}\ell_1(y)-n(d(x)+1)+\sum_{y\in N[x]}\ell_2(y)\\
=9n(m+1)/2-9n+9(n+1)/2=9(mn+1)/2=k^\prime,
\end{array}
$$
so $\ell$ is closed distance magic labeling of $C_{m}\boxtimes C_{n}$.
\end{proof}

Clearly the last lemma finishes the proof of Theorem \ref{conjecture}.
\end{proof}


\section{Circulant graphs}

The circulant graph $Ci(n,S)$, where $S\subset \{1,\dots,\lfloor n/2 \rfloor\}$ is the graph with vertex set $\{v_0,\ldots,v_{n-1}\}$, where $v_i$ and $v_j$ are adjacent if and only if $|i-j|\in S$. Observe that each closed distance magic labeling of $Ci(n,S)$ is exactly a $D$-distance magic labeling of a cycle $C_n$, where $D=S\cup \{0\}$. In particular, in \cite{ref_SimElv} Simanjuntak et al. have proved two following facts.

\begin{myproposition}[\cite{ref_SimElv}, Corollary $2$] For a positive integer $k$, the circulant graph $Ci(n,\{1,\ldots,k-1,k+1,\ldots,\left\lfloor\frac{n}{2}\right\rfloor\})$ is closed distance magic if and only if
$n= 4k$.
\end{myproposition}

\begin{myproposition}[\cite{ref_SimElv}, Theorem $6$]
For $n\geq 2k+2$ the circulant graph $Ci(n,\{1,\ldots,k\})$ is not closed distance magic.
\end{myproposition}

We generalize the last result by the following observation.

\begin{mylemma}\label{Syl1}
Let $G=Ci(n,\{c,2c,\ldots,kc\})$, where $k\geq 1$. If $G$ is closed distance magic, then $n=2kc$ or $n=(2k+1)c$.
\end{mylemma}

\begin{proof}
Suppose that the graph $G$ is closed distance magic with the closed  distance magic labeling $\ell$, then we obtain that $w(v_{kc})-w(v_{(k+1)c})=0$.
Hence for $n>2kc$ we obtain $\ell(v_{0})-\ell(v_{(2k+1)c})=0$, where the operation on the suffix is taken
modulo $n$. If $n>2kc$, this implies immediately that $(2k+1)c \equiv 0 \pmod {n}$ and in consequence $n=(2k+1)c$. On the other hand, by definition of $G$ we have $n\geq 2kc$ and the conclusion follows.
\end{proof}



Two following observations give sufficient conditions for the existence of a closed distance magic labeling of $G= Ci(n,\{c,2c,\ldots,kc\})$.

\begin{mylemma} \label{Syl2}
Let $G=Ci(n,\{c,2c,\ldots,kc\})$. If $n=2kc$ then $G$ is closed distance magic.
\end{mylemma}

\begin{proof}
Let $\ell(v_i)=i+1$, $\ell(v_{ck+i})=n-i$ for $i\in\{0,1,\ldots,ck-1\}$.
Obviously $\ell $ is a bijection and moreover for each $j\in\{0,\ldots,n-1\}$ we have $w(x_j)=k(n+1)$.
\end{proof}

\begin{mylemma} \label{Syl3}
Let $G=Ci(n,\{c,2c,\ldots,kc\})$ and $n=(2k+1)c$.  A graph $G$ is closed distance magic if and only if $c$ is odd.
\end{mylemma}

\begin{proof}
In the case when $c=1$, $G$ is a complete graph and thus also closed distance magic graph.
Hence we can focus on the case when $c\geq 2$. Notice that $G$ is $2k$-regular. By
Observation~\ref{cor_regular} there is no closed distance magic $2k$-regular graph $G$
with $n$ being even. Thus $c$ has to be odd.

Conversely, if $c$ is an odd integer, then there exists a magic $(2k+1,c)$-rectangle by
Theorem~\ref{rectangles}. Let $a_{i,j}$ be an $(i,j)$-entry of the $(2k+1,c)$-rectangle,
$1\leq i\leq 2k+1$ and $1\leq j\leq c$ . Define the labeling
$\ell \colon V(G) \rightarrow \{1,\ldots ,n\}$ as $\ell(v_{ic+j})=a_{i+1,j+1}$ for
$i\in \{0,\ldots ,2k\}$ and $j\in \{0,\ldots,c-1\}$. Obviously $\ell $ is a bijection
and moreover for each $t\in\{0,1,\ldots,n-1\}$ we have $w(v_t)=k'$.
\end{proof}

As an immediate consequence of Lemmas~\ref{Syl1}--\ref{Syl3} we obtain the following theorem:

\begin{mytheorem} Let $G=Ci(n,\{c,2c,\ldots,kc\})$. The graph $G$ is  closed distance magic if and only if either $n=2kc$ or $n=(2k+1)c$ and $c$ is odd.
\end{mytheorem}

Observe that if $n\equiv 0\pmod c$, then $Ci(2kc,\{c,2c,\ldots,kc\})\cong cK_{2k}$ and $Ci((2k+1)c,\{c,2c,\ldots,kc\})\cong cK_{2k+1}$. Thus we obtain the following corollary.

\begin{mycorollary}
Given $n\geq 2$ and $c\geq 1$, the union $cK_n$ is closed distance magic if and only if $n(c+1)\equiv 0\pmod 2$.
\end{mycorollary}

We finish this section with necessary conditions for $Ci(n,\{1,2,\dots,k-1,k+1\})$ to be closed distance magic. We know that the set of distinct eigenvalues of $G=Ci(n,S)$ is $\{\lambda_j=\sum_{s=2}^{|V(G)|}{a_s\omega^{(s-1)j}}\}$, $j\in\{0,\dots,n-1\}$, where $\omega=\exp(2\pi i/n)$ and $a_s$ is the $s$-th entry of the first row of the adjacency matrix of $G$ (see e.g. \cite{Big}, p. 16). This can be rewritten as $\{\lambda_j=2\sum_{s\in S}{\cos(2j\pi s/n)}\}$. Before we proceed, let us prove the following facts.\\

\begin{mylemma}\label{lem_solutionCos1}
The equation
$$
\sum_{s=1}^{k}{\cos(sx)}=-1/2
$$
has $2k$ solutions in the interval $[-\pi,\pi]$, namely $\pm 2j\pi/(2k+1), j\in\{1,\ldots,k\}.$
\end{mylemma}

\begin{proof} Observe that the above equation has at most $2k$ solutions in the interval
$[-\pi,\pi]$, as the highest multiple of $x$ is $k$ and all the multiples are integers.
On the other hand we have
$$
\begin{array}{l}
\sum_{s=1}^{k}{\cos(sx)}=\sum_{s=1}^{k}\frac{(e^{isx}+e^{-isx})}{2}=\frac{1}{2}\left(e^{ix}\frac{e^{ikx}-1}{e^{ix}-1}+e^{-ix}\frac{e^{-ikx}-1}{e^{-ix}-1}\right)=\\
=\frac{e^{ikx/2}-e^{-ikx/2}}{2i}\frac{e^{i(k+1)x/2}+e^{-i(k+1)x/2}}{2}\frac{2i}{e^{ix/2}-e^{-ix/2}}=\frac{\sin(kx/2)\cos((k+1)x/2)}{\sin(x/2)}=\\
=\frac{\sin((2k+1)x/2)}{2\sin(x/2)}-1/2.
\end{array}
$$
Now substituting $x$ with any of the numbers $2j\pi/(2k+1), j\in\{\pm 1,\ldots,\pm k\}$ we obtain
$$
\sin((2k+1)x/2)=\sin(j\pi)=0,
$$
and thus $\sum_{s=1}^{k}{\cos(sx)}=-1/2$ for every $\pm 2j\pi/(2k+1), j\in\{1,\ldots,k\}$.
\end{proof}

\begin{mylemma}\label{lem_solutionCos2}
The equation
$$
\sum_{s=1}^{k-1}{cos(sx)}+\cos((k+1)x)=-1/2
$$
has $2k+2$ solutions in the interval $[-\pi,\pi]$, namely $\pm 2j\pi/(2k+1), j\in\{1,\dots,k\}$, and $\pm \pi/3$.
\end{mylemma}

\begin{proof} Observe that the above equation has at most $2k+2$ solutions in the interval $[-\pi,\pi]$, as the highest multiple of $x$ is $k+1$ and all the multiples are integers. On the other hand we have
$$
\cos(\frac{2s(k+1)\pi}{2k+1})-\cos(\frac{2sk\pi}{2k+1})=-2\sin(s\pi) \sin\frac{s\pi}{2k+1}=0,
$$
and thus $\sum_{s=1}^{k-1}{cos(sx)}+\cos((k+1)x)=-1/2$ for every $\pm 2j\pi/(2k+1), j\in\{1,\ldots,k\}$ by Lemma \ref{lem_solutionCos1}. We have also
$$
\begin{array}{l}
\sum_{s=1}^{k-1}{cos(s\pi/3)}+\cos((k+1)\pi/3)=\\
=\frac{\sin((2k+1)\pi/6)}{2\sin(\pi/6)}-1/2-\cos(k\pi/3)+\cos((k+1)\pi/3)=-1/2,
\end{array}
$$
which finishes the proof.
\end{proof}



The following proposition gives the necessary conditions for $Ci(n,\{1,\dots,k-1,k+1\})$ to be closed distance magic.

\begin{myproposition}\label{prop_Ci2}
For given $n$ and $k$, $1<k\leq \lfloor(n-3)/2\rfloor$, let the multiplicity of $-1$ in $Sp(Ci(n,\{1,\dots,k-1,k+1\}))$ be $m$. Then $m=m_1+m_2$, where
$$m_1=2|\{t|1\leq t\leq k, nt\equiv 0\, ({\rm mod}\, {2k+1})\}|$$
and $m_2=2$ if $n\equiv 0\, ({\rm mod}\, 6)$ and $m_2=0$ otherwise. In particular, if the graph $Ci(n,\{1,\dots,k-1,k+1\})$ is closed distance magic, then $nt\equiv 0\, ({\rm mod}\, {2k+1})$ for some $t\in\{1,\ldots,k\}$ or $n\equiv 0\, ({\rm mod}\, 6)$.
\end{myproposition}

\begin{proof} Let $G\cong Sp(Ci(n,\{1,\dots,k-1,k+1\}))$ be closed distance magic graph. As $G$ is regular, by Corollary \ref{neces_Sp-1}
we have that $-1\in Sp(G)$ and hence $\sum_{s=1}^{k-1}{\cos(2j\pi s/n)}+\cos(2j\pi (k+1)/n)=-1/2$. Putting $x=2j\pi/n$ in Lemma \ref{lem_solutionCos2}
we obtain that $2j\pi/n=\pm 2t\pi/(2k+1)$ and so $j=\pm tn/(2k+1)$ for some $t\in\{1,\ldots,k\}$, or $2j\pi/n=\pm \pi/3$ and so $j=\pm n/6$. \end{proof}




\section{A solution to combinatorial problem forcing more closed distance magic graphs}

Let us consider graphs created in the following way: take the cycle $C_{k}$, exchange every vertex $v_{i}$ to a complete graph of some order (we will denote such a clique by $K[v_{i}]$) and join all the vertices of this complete graph to all the vertices of $K[v_{j}]$ for every $v_{j}$ being a neighbor of $v_{i}$ in $C_k$. In other words: every edge $v_{i}v_{j}$ of the original graph (cycle) becomes a complete graph $K_{p_{i}+p_{j}}$, where $p_{i}$ and $p_{j}$ are the orders of $K[v_{i}]$ and $K[v_{j}]$, respectively. Let $n=|V(G)|=p_{1}+\dots +p_{k}$.

In the special case where all the complete graphs are of the same order $p$, we obtain $G\cong C_k\boxtimes K_p$ at the end. Hence on one hand we try to generalize Theorems \ref{Beena} and \ref{magic_product_G_K} by more general graphs than strong products, while on the other side we specialize these two theorems as the components under discussion are cycles instead of general regular graphs.

The following observation gives necessary conditions for a graph defined as above to be closed distance magic (in some cases).

\begin{myobservation}
Let $G$ be a closed distance magic graph constructed as above. If $k\not\equiv 0\imod 3$, then the sum of labels in every clique $K[v_i]$ equals to ${\binom{{n+1}}{{2}}}/k$.
\end{myobservation}

\begin{proof} For any $v_{i}\in V(C)$, let us choose any vertex $v\in K[v_{i}]$. Clearly we have
$w(v)=\sum_{u\in K[v_{i-1}]\cup K[v_{i}]\cup K[v_{i+1}]}{\ell (u)}$. For any $x\in K[v_{i+1}]$
we have in turn $w(x)=\sum_{u\in K[v_{i}]\cup K[v_{i+1}]\cup K[v_{i+2}]}{\ell (u)}$. As $w(v)=w(x)$,
we have $\sum_{u\in K[v_{i-1}]}{\ell (u)}=\sum_{u\in K[v_{i+2}]}{\ell (u)}$ for any $i$ and finally
$\sum_{u\in K[v_{i}]}{\ell (u)}=\sum_{u\in K[v_{j}]}{\ell (u)}$ for any $i$ and $j$ since
$k\not\equiv 0\, ({\rm mod}\, 3)$. Hence the result follows. \end{proof}

If $k\equiv 0 \imod 3$, then the necessary conditions have more complicated form. As in this case $\sum_{u\in K[v_{i-1}]}{\ell (u)}=\sum_{u\in K[v_{i+2}]}{\ell (u)}$ does not imply that $\sum_{u\in K[v_{i}]}{\ell (u)}=\sum_{u\in K[v_{j}]}{\ell (u)}$ for arbitrary $i$ and $j$, we can only sate the following.

\begin{myobservation}
Let $G$ be a closed distance magic graph constructed as above. If $k\equiv 0\imod 3$, then the there exist three numbers $S_0$, $S_1$, $S_2$ such that the sum of labels in every clique $K[v_i]$ equals to $S_{i\imod 3}$ and $S_0+S_1+S_2=3{\binom{{n+1}}{{2}}}/k$.
\end{myobservation}

We focus only on the first case. The idea is then to find such a labeling of vertices, that the sum of the labels is the same in every clique. This is not always the necessary condition, but if we are able to find such a labeling, then we can also label any graph in which the role of $C_k$ is played by any regular graph. So, the problem can be reformulated as follows.

\begin{myproblem}
\label{problem}Let $n$ and $p_{1},\dots ,p_{k}$ be positive integers such that $p_{1}+\dots +p_{k}=n$. When is it possible to find a partition of the set $\{1,\dots ,n\}$ into $k$ sets $A_{1},\dots ,A_{k}$ such that $|A_{i}|=p_{i}$ and $\sum_{x\in A_{i}}{x}={\binom{{n+1}}{{2}}}/k$ for every $i\in\{1,\ldots,k\}$?
\end{myproblem}

As far as we know, this problem has not been studied too much so far.   It provides a combinatorial problem where the solution yields a closed distance magic labeling of a family of graphs as described above.  To illustrate it let $p_{1}=3$, $p_{2}=p_{3}=4$ and $p_{4}=5$ and consequently $n=16$. Sets $A_{1}=\{3,15,16\}$, $A_{2}=\{1,6,13,14\}$, $A_{3}=\{2,9,11,12\}$, $A_{4}=\{4,5,7,8,10\}$ provide the desired partition in this case. Observe that $C_4$ can be replaced by any other regular graph on 4 vertices, namely $K_4$ or $2K_2$.   On the other hand, if $p_{1}=2$, $p_{2}=p_{3}=4$ and $p_{4}=6$, then $\sum_{x\in A_{i}}{x}$ should be also equal to $34$, but $\sum_{x\in A_{1}}{x}\leq 31$ for each partition. Also if $k$ is even and $p$ odd where $p=p_{1}=\cdots =p_{k}$, we now that there is no solution. Namely the construction yields $(3p-1)$-regular graph $C_{k}\boxtimes K_{p}$, which is not closed distance magic by Observation \ref{cor_regular} and hence there cannot exist a partition from Problem \ref{problem}.

We conclude with two simple necessary conditions for Problem \ref{problem}.

\begin{myobservation}
If the partition from Problem \ref{problem} exists, then $n\equiv x \imod{2k}$ where $x\in \{0,-1\}$.
\end{myobservation}

\begin{proof} It follows from the fact that ${\binom{{n+1}}{{2}}}/k$ must be an integer.
\end{proof}

The next condition is as follows. There always must be enough large labels so that their sum can be at least as big as the desired sum of the labels in cliques.

\begin{myobservation}
Assume that $p_1,\dots,p_k$ are given in the non-decreasing order. Let $P_j=\sum_{i=1}^{j}{p_i}$. If the partition from Problem \ref{problem} exists, then for any $1\leq j\leq k$,
$$
\sum_{i=n-P_j+1}^{n}{i}\geq j{\binom{{n+1}}{{2}}}/k.
$$
\end{myobservation}

\begin{proof} The left side is the sum of $P_j$ largest labels, while the right side is the desired sum of the labels in $j$ smallest cliques.
\end{proof}

It is worth mentioning that Miller, Rodger and Simanjuntak in \cite{MRS}, while considering distance magic labelings (where the sums are taken over the open neighborhoods $N_{G}(x)$), proved that the above necessary conditions are also sufficient for complete $k$-partite graphs for $k\in\{2,3\}$.

\section*{Acknowledgments}
We are very grateful to two anonymous Referees for detailed remarks that allowed to improve our paper.



\begin{thebibliography}{99}

\bibitem{ACPT}
M. Anholcer, S. Cichacz, I. Peterin, A. Tepeh, \emph{Distance magic labeling
and two products of graphs.} Graphs Combin. (2014), accepted, DOI: 10.1007/s00373-014-1455-8.

\bibitem{AFK}
S. Arumugam, D. Froncek, N. Kamatchi, \emph{Distance Magic
Graphs---A Survey.} J. Indonesian
Math. Soc. Special Edition (2011), 11--26.

\bibitem{Be}
S. Beena, \emph{On $\Sigma$ and $\Sigma^{\prime}$ labelled
graphs.} Discrete Math. \textbf{309} (2009), 1783--1787.

\bibitem{Big}
N. Biggs, \emph{Algebraic Graph Theory.} Second Edition, Cambridge Mathematical Library, Cambridge University Press, 1996. 

\bibitem{ref_BroHae}
A.E. Brouwer, W.H. Haemers, \emph{Spectra of graphs.}
Springer, 2012.

\bibitem{Chat}
F. Chatelin, \emph{Eigenvalues of Matrices.} Classics in Applied Mathematics 71, SIAM, 2012. 

\bibitem{Cic}
S. Cichacz, \emph{Distance magic (r,t)-hypercycles},  Utilitas Math. (2013), accepted.

\bibitem{CicFro}
S. Cichacz, D. Froncek, \emph{Distance magic circulant graphs.} Preprint Nr MD 071 (2013), http://www.ii.uj.edu.pl/documents/12980385/26042491/MD\_71.pdf.

\bibitem{Fro}
D. Froncek, \emph{Ordered distance antimagic graphs and handicap incomplete tournaments.} Submitted.

\bibitem{Ga}
J.A. Gallian, \emph{A dynamic survey of graph labeling.} Elec.
J. Combin. \textbf{DS6}. http://www.combinatorics.org/Surveys/.

\bibitem{Hag}
T.R. Hagedorn, \emph{Magic rectangles revisited.} Discrete
Math. \textbf{207} (1999), 65--72.

\bibitem{IK}
R. Hammack, W. Imrich, S. Klav\v{z}ar, \emph{Handbook of Product
Graphs.} Second Edition, Discrete Mathematics and Its Applications, CRC
Press, Boca Raton, 2011.


\bibitem{Har1}
T. Harmuth, \emph{\"{U}ber magische Quadrate und \"{a}hniche
Zahlenfiguren.} Arch. Math. Phys. \textbf{66} (1881), 286--313.

\bibitem{Har2}
T. Harmuth,\emph{\"{U}ber magische Rechtecke mit ungeraden
Seitenzahlen.} Arch. Math. Phys. \textbf{66} (1881), 413--447.

\bibitem{MRS}
M. Miller, C. Rodger and R. Simanjuntak, \emph{Distance magic labelings of graphs},
Australasian Journal of Combinatorics \textbf{28} (2003), 305--315.

\bibitem{ONSl}
A. O'Neal, P. Slater, \emph{Uniqueness of vertex magic constants.}
Siam J. Discrete Math. \textbf{27} (2013), 708--716.

\bibitem{ONSl2}
A. O'Neal, P. Slater, \emph{An introduction to distance D magic graphs.} J. Indonesian
Math. Soc. Special Edition (2011), 89--107.

\bibitem{ref_SimElv}
R. Simanjuntak, M. Elviyenti, M. Nafie Jauhari, A. Sukmana Praja, I.A. Purwasih,
\emph{Magic labelings of distance at most $2$}, arXiv:1312.7633v1 [math.CO].


\end{thebibliography}
\end{document}